\documentclass[11pt,a4paper,twoside]{article}
\usepackage{amsthm, amsfonts,amsmath}

\newtheorem{Theorem}{Theorem}
\newtheorem{Lemma}{Lemma}
\newtheorem{Proposition}{Proposition}
\newtheorem{Corollary}{Corollary}
\newtheorem{Definition}{Definition}
\newtheorem{Remark}{Remark}
\newtheorem{Example}{Example}

\title{Weak Nearly $\mathcal S$- and Weak Nearly $\mathcal C$- Manifolds}

\author{Vladimir Rovenski
\footnote{Department of Mathematics, University of Haifa, Mount Carmel, 3498838 Haifa, Israel
\newline e-mail: {\tt vrovenski@univ.haifa.ac.il}
}}

\begin{document}

\date{}

\maketitle

\begin{abstract}
The recent interest of geometers in the $f$-structures of K.~Yano is motivated by the study of the dynamics of contact foliations, as well as their applications in theoretical physics.
Weak metric $f$-structures on a smooth manifold, recently introduced by the author and R. Wolak,
open a new perspective on the theory of classical structures.
In~the paper, we define structures of this kind, called weak nearly ${\cal S}$- and weak nearly ${\cal C}$- structures, study their geometry, e.g. their relations to Killing vector fields, and characterize weak nearly ${\cal S}$- and weak nearly ${\cal C}$- submanifolds in a weak nearly K\"{a}hler manifold.


\end{abstract}

\section{Introduction}
\label{sec:00-ns}

The~$f$-structure introduced by K.\,Yano~\cite{yan} on a smooth manifold $M^{2n+s}$ serves as a higher-dimensio\-nal analog
of {almost complex structures} ($s=0$) and {almost contact structures} ($s=1$).
This structure is defined by a (1,1)-tensor $f$ of rank $2n$ such that \mbox{$f^3 + f = 0$.}
The tangent bundle splits into two complementary subbundles: \mbox{$TM=f(TM)\oplus\ker f$.}
The~restriction of $f$ to the $2n$-dimensional distribution $f(TM)$ defines a {complex structure}.
The existence of the $f$-structure on $M^{2n+s}$ is equivalent to a reduction of the
structure group to $U(n)\times O(s)$; see~\cite{b1970}.
A submanifold $M$ of an almost complex manifold $(\bar M,J)$ that satisfies the condition $\dim(T_xM\cap J(T_xM))=const>0$
naturally possesses an $f$-structure; see~\cite{L-1969}.
An $f$-structure is a~special case of an {almost product structure}, defined by two complementary orthogonal distributions of a Riemannian mani\-fold $(M,g)$.
Foliations appear when one or both distributions are involutive.
An~interesting case occurs when the sub-bundle $\ker f$ is parallelizable, leading to a framed $f$-structure
for which the reduced structure group is $U(n)\times {\rm Id}_s$.
In this scenario, there exist vector fields $\{\xi_i\}_{1\le i\le s}$ (called Reeb vector fields) spanning $\ker f$
with dual 1-forms $\{\eta^i\}_{1\le i\le s}$, satisfying ${f}^2 = -{\rm Id} + \sum\nolimits_{\,i=1}^s {\eta^i}\otimes {\xi_i}$.
Compatible Riemannian metrics, i.e., 
\[
 g({f}X,{f}Y)=g(X,Y)-\sum\nolimits_{\,i=1}^s {\eta^i}(X)\,{\eta^i}(Y),
\]
exist on any framed $f$-manifold, and we~obtain the {metric $f$-structure}; see~\cite{b1970,CFF-1990,tpw-2014,Di-T-2006}.

To generalize concepts and results from almost contact geometry to metric $f$-mani\-folds, geometers have introduced and studied various broad classes of metric $f$-structu\-res.
 A~metric $f$-manifold is termed a ${\cal K}$-{manifold} if it is normal and $d \Phi=0$, where $\Phi(X,Y):=g(X,{f} Y)$.
Two important subclasses of ${\cal K}$-manifolds are {${\cal C}$-manifolds} if $d\eta^i=0$ and {${\cal S}$-manifolds} if $d\eta^i = \Phi$ for any~$i$;
see~\cite{b1970}.
Omitting the normality condition, we obtain almost ${\cal K}$-manifolds, almost ${\cal S}$-manifolds and almost ${\cal C}$-manifolds, e.g.,~\cite{CdT-2007,CFL-2020,Fitz-2011}.
The~distribution $\ker f$ of a ${\cal K}$-manifold is tangent to a $\mathfrak{g}$-foliation with flat totally geodesic leaves.
An~$f$-{K}-contact manifold is an almost ${\cal S}$-manifold, whose Reeb vector fields are Killing vector fields; the structure is intermediate between almost ${\cal S}$-structure and $S$-structure; see~\cite{Di-T-2006,Goertsches-2}.
Nearly ${\cal S}$- and nearly ${\cal C}$-manifolds $(M^{2n+s}, {f},\xi_i,\eta^i,g)$ are defined in the same spirit
as the nearly K\"{a}hler manifolds of A.~Gray~\cite{G-70}
by a constraint only on the symmetric part of $\nabla{f}$ --
starting from ${\cal S}$- and ${\cal C}$-manifolds
(e.g.,~\cite{BA-2019,AT-2017,rov-122,blair1976}):
\begin{equation*}
 (\nabla_X{f})X =
 \left\{\begin{array}{cc}
  g(f X, f X)\,\bar\xi +\bar\eta(X) f^2 X\,, & \mbox{nearly\ } {\cal S}-\mbox{manifolds}. \\
  0\,, & \mbox{nearly\ } {\cal C}-\mbox{manifolds}.
 \end{array}\right.
\end{equation*}

Here, $\bar\eta=\sum\nolimits_{\,i=1}^s\eta^i$ and $\bar\xi=\sum\nolimits_{\,i=1}^s\xi_i$.
These counterparts of nearly K\"{a}hler manifolds
play a key role in the classification of metric $f$-manifolds; see~\cite{b1970}.
The Reeb vector fields $\xi_i$ of nearly ${\cal S}$- and nearly ${\cal C}$-structures are unit Killing vector fields.
The~influence of constant-length Killing vector fields on Riemannian geometry has been studied by many authors, e.g.,~\cite{N-2021}.
The interest of geometers in $f$-structures is also motivated by the study of the dynamics of contact foliations.
Contact foliations generalize to higher dimensions the flow of
the Reeb vector field on contact manifolds, and ${\cal K}$-structures are a particular case of uniform $s$-contact structures; see~\cite{Alm-2024,Fil-2024}.
Dynamics and integration on $s$-cosymplectic manifolds are studied in~\cite{LSZ-2025};
they investigate the Lie integrability of $s$-evolution systems in this setting, and develop a Hamilton--Jacobi theory tailored to multi-time
Hamiltonian systems, both via symplectification techniques.

%

In~\cite{RWo-2,rst-43,rov-127}, we introduced and studied metric structures on a smooth manifold, see Definition~\ref{D-wK2},
which genera\-lize almost Hermitian, almost contact (e.g., Sasakian and cosymplectic) and $f$-structures.
 Such so-called ``weak'' structures (the complex structure on the contact distribution is replaced by a nonsingular skew-symmetric tensor)
allow us a new look at the theory of classical structures and find new applications.
A.~Einstein worked on various variants of Unified Field Theory,
more recently known as Non-symmetric Gravitational Theory (NGT), see~\cite{Mof}.
In this theory, the symmetric part $g$ of the basic tensor $G = g + F$ is associated with gravity, and the skew-symmetric one $F$ is associated with electromagnetism. The theory of weak metric structures is fully consistent with the skew-symmetric part of $G$; thus, it provides new tools for studying NGT.
S.~Ivanov and M.~Zlatanovi\' c developed NGT with li\-near connections of totally skew-symmetric torsion
and gave examples with the skew-symmetric part $F$ of the tensor $G$ obtained using an almost contact metric structure; see~\cite{IZ1}.
In~\cite{ZR1}, the author and M.\,Zlatanovi\'c were the first to apply weak
metric structures to
NGT of totally skew-symmetric torsion with tensor $F(X,Y)=g(X,fY)$ of constant~rank.

In this paper, we define and study new structures of this kind, generalizing nearly ${\cal S}$- and nearly ${\cal C}$-structures.
Section~\ref{sec:01-ns}, following the Introduction,
recalls some results regarding
weak nearly K\"{a}hler manifolds (generalizing nearly K\"{a}hler manifolds) and weak metric $f$-manifolds.
Section~\ref{sec:02-ns} introduces weak nearly ${\cal S}$- and weak nearly ${\cal C}$-structures and studies their geometry.
Section~\ref{sec:03-ns} characterizes weak nearly ${\cal C}$- and weak nearly ${\cal S}$-submanifolds in weak nearly K\"{a}hler manifolds and proves that a weak nearly ${\cal C}$-manifold with parallel Reeb vector fields is locally the Riemannian product of a Euclidean space
and a weak nearly K\"{a}hler manifold.
The~proofs use the properties of new tensors, as well as classical~constructions.

\section{Preliminaries}
\label{sec:01-ns}

Here, we review some
results; see~\cite{RWo-2,rst-43,rov-127}.
Nearly K\"{a}hler {manifolds} $(M,J,g)$ were defined by A.~Gray~\cite{G-70} using the condition that only the symmetric part of $\nabla J$ vanishes,
where $\nabla$ is the Levi-Civita connection,
in contrast to the K\"{a}hler case, where $\nabla J=0$.
Several authors studied the problem of finding and classifying parallel skew-symmetric 2-tensors
(other than almost-complex structures) on a Riemannian manifold, e.g.,~\cite{H-2022}.

\begin{Definition}\label{D-wK2}\rm

A Riemannian manifold $(M, g)$ of even dimension equipped with a skew-symmetric (1,1)-tensor ${f}$
such that the tensor ${f}^{\,2}$ is negative-definite is called a \textit{weak Hermitian mani\-fold}.
Such $(M, {f}, g)$ is called a \textit{weak K\"{a}hler manifold} if $\nabla{f}=0$.
A~weak Hermitian manifold is called a \textit{weak nearly K\"{a}hler manifold}
if
\begin{equation}\label{Eq-NS-9}
 (\nabla_X{f})Y + (\nabla_Y{f})X=0\quad (X,Y\in\mathfrak{X}_M).
\end{equation}
A~\textit{weak metric $f$-structure} on a smooth manifold $M^{2n+s}$ $(n,s>0)$ is a set $({f},Q,{\xi_i},{\eta^i},g)$, where
${f}$ is a skew-symmetric $(1,1)$-tensor of rank $2\,n$, $Q$ is a self-adjoint nonsingular $(1,1)$-tensor,
${\xi_i}\ (1\le i\le s)$ are orthonormal
vector fields, ${\eta^i}$ are dual 1-forms, and $g$ is a~Riemannian metric on $M$, satisfying
\begin{align}\label{2.1}
 {f}^2 = -Q + \sum\nolimits_{\,i=1}^s {\eta^i}\otimes {\xi_i},\quad {\eta^i}({\xi_j})=\delta^i_j,\quad
 Q\,{\xi_i} = {\xi_i}, \\
\label{2.2}
 g({f} X,{f} Y)= g(X,Q\,Y) -\sum\nolimits_{\,i=1}^s {\eta^i}(X)\,{\eta^i}(Y)\quad (X,Y\in\mathfrak{X}_M).
\end{align}
In this case, $(M^{2n+s}, {f},Q,{\xi_i},{\eta^i},g)$ is called a \textit{weak metric $f$-manifold}.
\end{Definition}

The geometric meaning of \eqref{Eq-NS-9} is the same as in the classical case: geodesics are $f$-planar curves.
A curve $\gamma$ is \textit{f-planar} if the section $\dot\gamma\wedge f\dot\gamma$ is parallel along the curve.
A~framed weak $f$-manifold (i.e., only \eqref{2.1} holds) admits a compatible metric (i.e., also \eqref{2.2} holds) if ${f}$
in \eqref{2.1} has a skew-sym\-metric representation, i.e., for any $x\in M$ there exists a~frame $\{e_i\}$ on a neighborhood $U_x\subset M$,
for which ${f}$ has a skew-symmetric matrix.

\begin{Example}\rm
Take $k>1$ almost Hermitian manifolds $(M_j, f_j, g_j)$.
The~Riemannian pro\-duct $\prod_{\,j=1}^k (M_j, \lambda_j^{1/2}\,f_j, g_j)$,
where $\lambda_j>0$ are different constants, is a weak almost Hermitian manifold with $Q=\bigoplus_{\,j}\lambda_j\,{\rm Id}_{\,j}$.
We~call $\prod_{\,j} (M_j, \lambda_j^{1/2}\,f_j, g_j)$ a $(\lambda_1,\ldots,\lambda_k)$-\textit{weighed product of almost Hermitian manifolds} $(M_j, f_j, g_j)$; see~\cite{rz-2025}.
The $(\lambda_1,\ldots,\lambda_k)$-weighed product of (nearly) K\"{a}hler manifolds is a weak (nearly) K\"{a}hler mani\-fold.
A~ne\-arly K\"{a}hler manifold of dimension $\le 4$ is a K\"{a}hler manifold; see~\cite{G-70}.
The~unit sphere $S^6$ in the set of purely imaginary Cayley numbers admits a strictly nearly K\"{a}hler structure.
The~classification of weak nearly K\"{a}hler manifolds in dimensions $\ge 4$ is an open problem.
The $(\lambda_1,\lambda_2)$-weighed pro\-ducts of 2-dimensional K\"{a}hler manifolds are 4-dimensional weak nearly K\"{a}hler manifolds.
The $(\lambda_1,\lambda_2,\lambda_3)$-weighed products of 2-dimensional K\"{a}hler manifolds
and $(\lambda_1,\lambda_2)$-weighed products of 2- and 4-dimensional K\"{a}hler manifolds
are  6-dimensional weak nearly K\"{a}hler manifolds, and similarly for
dimensions $>6$.
\end{Example}

Putting $Y=\xi_j$ in \eqref{2.2}, and using ${\eta^i}({\xi_j})=\delta^i_j$, we get
\begin{align}\label{2.2-eta}
 \eta^j(X) = g(X,\xi_j);
\end{align}
thus, ${\xi_j}$ is orthogonal to the distribution ${\mathcal D}=\bigcap_{\,i=1}^s \ker{\eta^i}$.
For a more intuitive under\-standing of the role of $Q$ in the $f$-structure, we explain the following properties:
\[
 {f}\,{\xi_i}=0,\quad {\eta^i}\circ{f}=0,\quad \eta^i\circ Q=\eta^i,\quad [Q,\,{f}]=0 .
\]
By \eqref{2.1}, $f^2\xi_i=0$ is true. From this and \eqref{2.1}, we get
 $f^3 + fQ = 0$.
By this, $Q\xi_i=\xi_i$ and $f^2\xi_i=0$ we get $0=-f^3\xi_i=fQ\xi_i=f\xi_i$.
By $f\xi_i=0$, \eqref{2.2-eta}, and the skew-symmetry of $f$, we get $\eta^i(fX)=g(fX,\xi_i)=-g(X,f\xi_i)=0$.
From this and condition ${\rm rank}\,f=2n$, we conclude that $f$ the distribution ${\mathcal D}$
of a weak metric $f$-structure is ${f}$-invariant, ${\mathcal D}=f(TM)$ and $\dim{\mathcal D}=2\,n$.
By this and $f^3 + fQ = 0$, we get $f^3X=f^2(fX)=-QfX$; hence, $f^3 + Qf = 0$.
This and $f^3 + fQ = 0$ yield $fQ=Qf$. By symmetry of $Q$ and $Q\xi_i=\xi_i$, we get $\eta^i(QX)=g(QX,\xi_i)=g(X,Q\xi_i)=g(X,\xi_i)=\eta^i(X)$.
Therefore, $TM$ splits as complementary orthogonal sum of
${\mathcal D}$ and~$\ker f$.
 A weak metric $f$-structure $({f},Q,{\xi_i},{\eta^i},g)$ is said to be {normal} if the following tensor is zero:
\begin{align*}
 {\cal N}^{\,(1)}(X,Y) = [{f},{f}](X,Y) + 2\sum\nolimits_{\,i=1}^s d{\eta^i}(X,Y)\,{\xi_i}\quad (X,Y\in\mathfrak{X}_M).
\end{align*}
The Nijenhuis torsion
of a (1,1)-tensor ${S}$ and the
derivative of a 1-form ${\omega}$ are given~by
\begin{align*}
\notag
 & [{S},{S}](X,Y) = {S}^2 [X,Y] + [{S} X, {S} Y] - {S}[{S} X,Y] - {S}[X,{S} Y]\quad (X,Y\in\mathfrak{X}_M), \\
 & d\omega(X,Y) = (1/2)\,\{X({\omega}(Y)) - Y({\omega}(X)) - {\omega}([X,Y])\}\quad (X,Y\in\mathfrak{X}_M).
\end{align*}
Using the Levi-Civita connection $\nabla$ of $g$, one can rewrite $[S,S]$ as
\begin{align}\label{4.NN}
 [{S},{S}](X,Y) = ({S}\nabla_Y{S} - \nabla_{{S} Y}{S}) X - ({S}\nabla_X{S} - \nabla_{{S} X}{S}) Y .
\end{align}
The {fundamental $2$-form} $\Phi$ on $(M^{2n+s}, {f},Q,\xi_i,\eta^i,g)$ is defined by
\[
 \Phi(X,Y)=g(X,{f} Y)\quad
  (X,Y\in\mathfrak{X}_M).
\]

\begin{Proposition}
A weak metric $f$-structure with condition ${\cal N}^{\,(1)}=0$ satisfies
\begin{align*}
 & \pounds_{{\xi_i}}{f} = d{\eta^j}({\xi_i}, \cdot)
 = 0,\\
 & d{\eta^i}({f} X,Y) - d{\eta^i}({f} Y,X)
= \frac12\,\eta^i([\widetilde QX, fY]),\\
 & \nabla_{\xi_i}\,\xi_j\in{\cal D},\quad  [X,\xi_i]\in{\cal D}\quad (1\le i,j\le s,\ X\in{\cal D}).
\end{align*}
Moreover, $\nabla_{\xi_i}\,\xi_j+\nabla_{\xi_j}\,\xi_i=0$, that is, $\ker f$ defines a totally geodesic distribution.
\end{Proposition}

These tensors
on a weak metric $f$-manifold are well known in the classical theory:
\begin{align*}
 {\cal N}^{\,(2)}_i(X,Y) &:= (\pounds_{{f} X}\,{\eta^i})(Y) - (\pounds_{{f} Y}\,{\eta^i})(X)
 =2\,d{\eta^i}({f} X,Y) - 2\,d{\eta^i}({f} Y,X) ,  \\
 {\cal N}^{\,(3)}_i(X) &:= (\pounds_{{\xi_i}}{f})X
 = [{\xi_i}, {f} X] - {f} [{\xi_i}, X],\\
 {\cal N}^{\,(4)}_{ij}(X) &:= (\pounds_{{\xi_i}}\,{\eta^j})(X)
 = {\xi_i}({\eta^j}(X)) - {\eta^j}([{\xi_i}, X])
 = 2\,d{\eta^j}({\xi_i}, X) .
\end{align*}

\begin{Example}\rm
Let $M^{2n+s}(f,Q,\xi_i,\eta^i)$ be a weak framed $f$-manifold.
Consider the product manifold $\bar M = M^{2n+s}\times\mathbb{R}^s$,
where $\mathbb{R}^s$ is a Euclidean space with a basis $\partial_1,\ldots,\partial_s$,
and define tensors $J$ and $\bar Q$ on $\bar M$ putting
 $J(X, \sum\nolimits_{\,i=1}^s a^i\partial_i) = (fX - \sum\nolimits_{\,i=1}^s a^i\xi_i,\, \sum\nolimits_{\,j}\eta^j(X)\partial_j)$
 and
 $\bar Q(X, \sum\nolimits_{\,i=1}^s a^i\partial_i) = (QX,\, \sum\nolimits_{\,i=1}^s a^i\partial_i)$
for $a_i\in C^\infty(M)$.
It can be shown that $J^{\,2}=-\bar Q$.
 The~tensors ${\cal N}^{\,(2)}_i, {\cal N}^{\,(3)}_i, {\cal N}^{\,(4)}_{ij}$ appear when~we derive the integrability condition $[J, J]=0$
and express the normality condition ${\cal N}^{\,(1)}=0$ for~$(f,Q,\xi_i,\eta^i)$.
\end{Example}

Define a ``small'' (1, 1)-tensor $\tilde Q:=Q - {\rm Id}$
and note that $[\widetilde{Q},{f}]=0$ and $\eta^i\circ\widetilde Q=0$.
The following new tensor
(vani\-shing at $\widetilde Q=0$)
\begin{align*}
 & {\cal N}^{\,(5)}(X,Y,Z) := {f} Z\,(g(X, \widetilde QY)) - {f} Y\,(g(X, \widetilde QZ)) \\
 & + g([X, {f} Z], \widetilde QY) - g([X,{f} Y], \widetilde QZ) + g([Y,{f} Z] -[Z, {f} Y] - {f}[Y,Z],\ \widetilde Q X),
\end{align*}
which supplements the
sequence
${\cal N}^{\,(1)},{\cal N}^{\,(2)}_i,{\cal N}^{\,(3)}_i,{\cal N}^{\,(4)}_{ij}$,
is needed to study the weak metric $f$-structure.
We express the covariant derivative of $f$ using a new tensor ${\cal N}^{\,(5)}$:
\begin{align*}
 & 2\,g((\nabla_{X}{f})Y,Z) = 3\,d\Phi(X,{f} Y,{f} Z) - 3\, d\Phi(X,Y,Z) + g({\cal N}^{\,(1)}(Y,Z),{f} X)\notag\\
 & +\sum\nolimits_{\,i=1}^s \big({\cal N}^{\,(2)}_i(Y,Z)\,\eta^i(X) + 2\,d\eta^i({f} Y,X)\,\eta^i(Z) - 2\,d\eta^i({f} Z,X)\,\eta^i(Y)\big)
 + {\cal N}^{\,(5)}(X,Y,Z),
\end{align*}
where
the derivative of a $2$-form $\Phi$ is given by
\begin{align*}
 3\,d\Phi(X,Y,Z) &= X\,\Phi(Y,Z) + Y\,\Phi(Z,X) + Z\,\Phi(X,Y) \notag\\
 &-\Phi([X,Y],Z) - \Phi([Z,X],Y) - \Phi([Y,Z],X).
\end{align*}
Note that the above equality yields
\begin{align}\label{E-3.3}
 3\,d\Phi(X,Y,Z) = (\nabla_X\,\Phi)(Y,Z)+(\nabla_Y\,\Phi)(Z,X)+(\nabla_Z\,\Phi)(X,Y).
\end{align}
For particular values of
${\cal N}^{\,(5)}$, we get
${\cal N}^{\,(5)}(\xi_i,\xi_j,Z) = {\cal N}^{\,(5)}(\xi_i,Y,\xi_j)=0$ and
\begin{align*}
\nonumber
 {\cal N}^{\,(5)}(X,\xi_i,Z) & = -{\cal N}^{\,(5)}(X, Z, \xi_i) = g( {\cal N}^{\,(3)}_i(Z),\, \widetilde Q X),\\
\nonumber
 {\cal N}^{\,(5)}(\xi_i,Y,Z) &= g([\xi_i, {f} Z], \widetilde QY) -g([\xi_i,{f} Y], \widetilde QZ).
\end{align*}

\begin{Definition}\rm
A weak metric $f$-structure is called
a \textit{weak almost ${\cal K}$-structure} if $d\Phi=0$.
We define its two subclasses as follows: 
\begin{enumerate}
\item[(i)] A \textit{weak almost ${\cal C}$-structure} if $\Phi$ and $\eta^i\ (1\le i\le s)$ are closed forms;
\item[(ii)] A \textit{weak almost ${\cal S}$-structure}
\end{enumerate}
if
the following is valid:
\begin{align}\label{2.3}
 \Phi=d{\eta^1}=\ldots =d{\eta^s} \quad
 ({\rm hence,}\ d\Phi=0).
\end{align}
Adding the normality condition,
 we get
\textit{weak ${\cal K}$-},
\textit{weak ${\cal C}$-}, and \textit{weak ${\cal S}$-structures},
respectively.
A \textit{weak $f$-{\rm K}-contact structure} is
a weak almost ${\cal S}$-structure, whose structure vector fields ${\xi_i}$ are Killing,~i.e., the tensor
 $(\pounds_{{\xi_i}}\,g)(X,Y)
  = g(\nabla_Y\,{\xi_i}, X) + g(\nabla_X\,{\xi_i}, Y)$
vanishes.
For $s=1$, weak (almost) ${\cal C}$- and weak (almost) ${\cal S}$-manifolds reduce to weak (almost) cosymplectic manifolds and weak (almost) Sasakian manifolds, respectively.
\end{Definition}

\begin{Remark}\rm
The almost ${\cal S}$-structu\-re is also called an $f$-contact structure, e.g.,~\cite{rst-43};
then, the {${\cal S}$-structure} can be regarded as a normal $f$-contact structure.
%
\end{Remark}

\begin{Example}\label{Ex-C-S}\rm
(i) To construct a weak
metric $f$-structure
$({f}, Q,\xi_i,\eta^i,g)$
on
the Riemannian product $M=\bar M\times\mathbb{R}^s$
of a weak almost Hermitian
manifold $(\bar M, \bar{f}, \bar g)$
with $\Omega(X,Y)=\bar g(X, \bar{f}Y)$ and a Euclidean space $(\mathbb{R}^s, dy^2)$, we take any point $(x, y)$ of $M$ and set
\[
 \xi_i = (0, \partial_{\,y^i}),\ \
 \eta^i =(0, dy^i),\ \
 {f}(X, \partial_{\,y^i}) = (\bar{f} X, 0),\ \
 Q(X, \partial_{\,y^i}) = (-\bar{f}^{\,2} X, \partial_{\,y^i}),
\]
where $X\in T_x\bar M$. Note that $\nabla f=0$ if and only if $\overline\nabla\bar f=0$.
On the other hand, $\overline\nabla\bar f=0$ if and only if $d\Omega=0$,
see \eqref{E-3.3} with $\Phi=\Omega$,
i.e., $(M,\Omega)$ is a symplectic manifold.

(ii) For a weak ${\cal C}$-structure,
we obtain
 $g((\nabla_{X}{f})Y,Z) = \frac{1}{2}\,{\cal N}^{\,(5)}(X,Y,Z)$.
A~weak metric $f$-structure
with conditions $\nabla{f}=0$ and
 $g([\xi_i,\xi_j],\xi_k)=0$
is a~weak ${\cal C}$-structure with the property ${\cal N}^{\,(5)}=0$.
For a weak ${\cal S}$-structure,
we get
\begin{align*}
 g((\nabla_{X}{f})Y,Z) = g(fX, fY)\,\bar\eta(Z) - g(fX, fZ)\,\bar\eta(Y)
 +\frac{1}{2}\, {\cal N}^{\,(5)}(X,Y,Z);
\end{align*}
$\xi_i$ are Killing vector fields and $\ker f$ defines
a Riemannian totally geodesic foliation.
In particular, for an ${\cal S}$-structure, we have
\begin{align}\label{3.1A}
 (\nabla_{X}{f})Y =  g(fX,fY)\,\bar\xi + \bar\eta(Y) f^2X\,.
\end{align}
\end{Example}

For a weak almost ${\cal K}$-structure (and its special cases, a weak almost ${\cal S}$-structure and a weak almost ${\cal C}$-structure), the distribution $\ker f$ is involutive (tangent to a foliation).
Moreover, weak almost ${\cal S}$- and weak almost ${\cal C}$-structures satisfy the following
conditions (trivial for~$s=1$):
\begin{align}
\label{E-30b-xi}
 [\xi_i, \xi_j] & =0, \\
\label{E-30-xi}
 g(\nabla_{X}\,\xi_i,\ \xi_j) & = 0\quad (X\in\mathfrak{X}_M)
\end{align}
for $1\le i,j \le s$.
The~following condition
is a corollary of \eqref{E-30-xi}:
\begin{align}\label{E-30c-xi}
 \eta^k(\nabla_{\xi_i}\,\xi_j)=0\quad(1\le i,j,k \le s) .
\end{align}
By~\eqref{E-30b-xi}, the distribution $\ker f$ of weak almost ${\cal S}$- and a weak almost ${\cal C}$-manifolds
is tangent to a $\mathfrak{g}$-foliation with an abelian Lie~algebra.

\begin{Remark}[\cite{AM-1995}]\rm
Let $\mathfrak{g}$ be a Lie algebra of dimension $s$.
A foliation of dimen\-sion $s$ on a smooth connected manifold $M$ is called a $\mathfrak{g}$-\textit{foliation} if there exist
complete vector fields $\xi_1,\ldots,\xi_s$ on $M$ which, when restricted to each leaf,
form a parallelism of this submanifold
with a Lie algebra isomorphic to~$\mathfrak{g}$.
\end{Remark}

\section{Main Results}
\label{sec:02-ns}

In this section, weak nearly ${\cal S}$- and weak nearly ${\cal C}$-structures are defined and studied; some of the statements generalize the results in~\cite{AT-2017,blair1976, rov-122}.

The restriction on the symmetric part of \eqref{3.1A} gives the following.

\begin{Definition}\label{D-w-nearly-SC}\rm
A weak metric $f$-manifold
is called a \textit{weak nearly ${\cal S}$-manifold} if
\begin{equation}\label{E-nS-01}
 (\nabla_X {f})Y + (\nabla_Y {f})X = 2\,g(fX,fY)\,\bar\xi + \bar\eta(X) f^2Y + \bar\eta(Y) f^2X
\end{equation}
for all $X,Y\in\mathfrak{X}_M$.
A weak metric $f$-manifold is called a \textit{weak nearly ${\cal C}$-manifold} if
\begin{equation}\label{E-nS-01b}
 (\nabla_X{f})Y + (\nabla_Y{f})X = 0.
\end{equation}
\end{Definition}

\begin{Example}\rm
Let a Riemannian manifold $(M^{2n+s},g)$ admit two nearly ${\cal S}$-structures (or, nearly ${\cal C}$-structures)
$M^{2n+s}({f}_k,Q,{\xi_i},{\eta^i},g)\ (k=1,2)$
 with common Reeb vector fields $\xi_i$ and
one-forms $\eta^i= g(\xi_i,\,\cdot)$.
Suppose that ${f}_1\ne{f}_2$ are such that $\psi:={f}_1{f}_2+{f}_2{f}_1\ne0$.
Then, ${f} := (\cos t){f}_1 + (\sin t){f}_2$ for small $t>0$
satisfies \eqref{E-nS-01} (and \eqref{E-nS-01b}, respectively) and
\[
 f^2 = -{\rm Id}+(\sin t\cos t)\,\psi + \sum\nolimits_{\,i=1}^s {\eta^i}\otimes {\xi_i}.
\]
Thus, $({f},Q,\xi_i,\eta^i,g)$ is a weak nearly ${\cal S}$-structure (and weak nearly ${\cal C}$-structu\-re, respectively) on $M^{2n+s}$ with
$Q={\rm Id}-(\sin t\cos t)\,\psi$.
\end{Example}

The following condition is trivial when $Q={\rm Id}_{\,TM}$:
\begin{equation}\label{E-nS-10}
  (\nabla_X\,Q)Y=0\quad (X,Y\in\mathfrak{X}_M,\ Y\perp\ker f).
\end{equation}
Using
\eqref{E-nS-10}, we have
\begin{align*}
 (\nabla_X\,Q)Y = \sum\nolimits_{\,i=1}^s {\eta^i}(Y)(\nabla_X\,Q)\xi_i
 = -\sum\nolimits_{\,i=1}^s \eta^i(Y)\,\widetilde Q\nabla_X\,\xi_i
 \quad (X,Y\in\mathfrak{X}_M).
\end{align*}

\begin{Example}\label{Ex-02}\rm
To construct a weak (nearly) $\cal C$-structure $({f}, Q,\xi_i,\eta^i,g)$ on the Riemannian product $M=\bar M\times\mathbb{R}^s$
of a weak (nearly) K\"{a}hler manifold $(\bar M, \bar{f}, \bar g)$ and a Euclidean space $(\mathbb{R}^s, dy^2)$, we take any point $(x, y)$ of $M$ and set
\[
 \xi_i = (0, \partial_{\,y^i}),\ \
 \eta^i =(0, dy^i),\ \
 {f}(X, \partial_{\,y^i}) = (\bar{f} X, 0),\ \
 Q(X, \partial_{\,y^i}) = (-\bar{f}^{\,2} X, \partial_{\,y^i}),
\]
as in Example~\ref{Ex-C-S}(i).
Note that if $\overline\nabla_X\,(\bar{f}^2)=0\ (X\in T\bar M)$, then \eqref{E-nS-10} holds.
\end{Example}


The following result
opens new applications to Killing vector fields.

\begin{Proposition}\label{Prop-1}
Both on a weak nearly ${\cal S}$-manifold and a weak nearly ${\cal C}$-manifold satisfying \eqref{E-30b-xi} and \eqref{E-30c-xi},
the distribution $\ker f$ defines a flat totally geodesic foliation;
moreover, if conditions \eqref{E-30-xi} and \eqref{E-nS-10} hold, then the vector fields $\xi_i$ are~Killing.
\end{Proposition}

\begin{proof}
 Putting $X=\xi_j$ and $Y=\xi_k$ in \eqref{E-nS-01} or \eqref{E-nS-01b}, we find $(\nabla_{\xi_j}{f})\xi_k+(\nabla_{\xi_k}{f})\xi_j=0$; hence, ${f} \big(\nabla_{\xi_j}\,\xi_k+\nabla_{\xi_k}\,\xi_j\big)=0$.
Applying ${f}$ to this and using \eqref{2.1}, we obtain
\begin{align*}
 0={f}^2\big(\nabla_{\xi_j}\,\xi_k+\nabla_{\xi_k}\,\xi_j\big)
 = -Q\big(\nabla_{\xi_j}\,\xi_k+\nabla_{\xi_k}\,\xi_j\big) +\sum\nolimits_{\,i=1}^s \eta^i\big(\nabla_{\xi_j}\,\xi_k+\nabla_{\xi_k}\,\xi_j\big)\xi_i.
\end{align*}
Since
the (1,1)-tensor
$Q$ is nonsingular and \eqref{E-30c-xi} is true, we get
 $\nabla_{\xi_j}\,\xi_k+\nabla_{\xi_k}\,\xi_j=0$.
Combining this with $\nabla_{\xi_j}\,\xi_k - \nabla_{\xi_k}\,\xi_j=0$, see \eqref{E-30b-xi}, yields
\begin{align}\label{E-nS-xixi}
 \nabla_{\xi_j}\,\xi_k=0\quad(1\le j,k\le s);
\end{align}
hence, $\ker f$ defines a flat totally geodesic foliation.
Next, using \eqref{E-nS-xixi} we calculate
\begin{align}\label{E-nS-xieta}
 \nabla_{\xi_i}\,\eta^j = g(\nabla_{\xi_i}\,\xi_j,\, \cdot) = 0 .
\end{align}
Using \eqref{E-30-xi} and \eqref{E-nS-xixi}, we obtain
\[
 (\pounds_{\xi_j}\,g)(\xi_k,\,\cdot) = g(\nabla_{\xi_j}\,\xi_k,\, \cdot) = 0.
\]
Taking the $\xi_j$-derivative of \eqref{2.2} and using \eqref{E-nS-10}
and $\nabla_{\xi_j}\,\eta^i=0$,
we find (for $Y\perp\ker f$)
\begin{eqnarray*}
 & g((\nabla_{\xi_j}{f})X, {f} Y) + g({f} X, (\nabla_{\xi_j}\,{f})Y) = \nabla_{\xi_j}\,g({f} X, {f} Y) \\
 & = g(X, (\nabla_{\xi_j}\,Q)Y) + \sum_{\,i=1}^s \big\{(\nabla_{\xi_j}\,\eta^i)(X)\,\eta^i(Y) + \eta^i(X) (\nabla_{\xi_j}\,\eta^i)(Y)\big\} = 0.
\end{eqnarray*}
For a weak nearly ${\cal S}$-manifold, using \eqref{E-nS-01}, \eqref{E-30-xi}, and $\eta\circ \widetilde Q=0$ yields
\begin{eqnarray*}
 &&\quad g((\nabla_{\xi_j}{f})X, {f} Y) + g({f} X, (\nabla_{\xi_j}\,{f})Y) \\
 && = -g((\nabla_X {f})\xi_j, {f} Y) - g({f} X, (\nabla_Y {f})\xi_j) + g( f^2X,\, {f} Y) + g( f^2Y,\, {f} X) \\
 && = -g(\nabla_X\,\xi_j, {f}^2 Y) - g({f}^2 X, \nabla_Y\,\xi_j)
 = g(\nabla_X\,\xi_j, Q Y) + g(Q X, \nabla_Y\,\xi_j) \\
 &&  = g(\nabla_X\,\xi_j, Y) + g(X, \nabla_Y\,\xi_j)
    + g(\nabla_X\,\xi_j, \widetilde Q Y) + g(\widetilde Q X, \nabla_Y\,\xi_j) \\
 && = (\pounds_{\xi_j}\,g)(X, Y) -g(\xi_j, (\nabla_X\widetilde Q) Y) - g((\nabla_Y\widetilde Q) X, \xi_j)
 = (\pounds_{\xi_j}\,g)(X, Y).
\end{eqnarray*}
Here, we used $g(\xi_j, (\nabla_X\widetilde Q) Y)=0$.
For a weak nearly ${\cal C}$-manifold, using \eqref{E-nS-01b}
yields
\begin{align}
\label{Eq-normal-5}
 (\pounds_{\xi_j}\,g)(X, Y)
 =g((\nabla_{\xi_j}{f})X, {f} Y) + g({f} X, (\nabla_{\xi_j}\,{f})Y) = 0.
\end{align}
From \eqref{Eq-normal-5},
 for both cases we obtain
 $\pounds_{\xi_j}\,g=0$,
i.e., $\xi_j$ is a Killing vector field.
\end{proof}

\begin{Remark}\rm
Note that even for a nearly ${\cal S}$-manifold without conditions \eqref{E-30b-xi} and \eqref{E-30-xi}, the vector fields $\xi_i\ (1\le i\le s)$ are not Killing; see Corollary 1 in~\cite{AT-2017}.
\end{Remark}

\begin{Theorem}
There are no weak nearly ${\cal C}$-manifolds with conditions \eqref{E-30b-xi}, \eqref{E-30-xi}, and \eqref{E-nS-10} which
satisfy $\Phi=d{\eta^1}=\ldots =d{\eta^s}$; see \eqref{2.3}.
\end{Theorem}

\begin{proof}
Suppose that our weak nearly ${\cal C}$-manifold satisfies \eqref{2.3}.
Since also $\xi_i$ are Killing vector fields (see Proposition~\ref{Prop-1}), $M$ is a weak $f$-K-contact manifold.
By~Theorem~1 in~\cite{rov-127}, the following~holds:
\begin{align}\label{E-nabla-xi}
\nabla\xi_i=-{f}\quad(1\le i\le s).
\end{align}
By Proposition~6 in~\cite{rov-127}, the $\xi$-sectional curvature of a weak $f$-K-contact manifold is positive, i.e., $K(\xi_i,X)>0\ (X\perp\ker f)$.
Thus, for any nonzero vector $X\perp\ker f$, using \eqref{E-nS-01b} and \eqref{E-nabla-xi}, we~get
\begin{align*}
 & 0< K(\xi_i,X) = g(\nabla_{\xi_i}\nabla_X \xi_i -\nabla_X\nabla_{\xi_i}\xi_i -\nabla_{[\xi_i,X]} \xi_i, X) \\
 &\quad = g( -(\nabla_{\xi_i}{f})X +{f}^2 X, X)
 = g( (\nabla_X{f})\xi_i, X) -g({f} X, {f} X) \\
 &\quad = - g( {f}\nabla_X\xi_i, X) +g({f}^2 X,  X)
   =  2\,g( {f}^2 X, X) .
\end{align*}
This contradicts the following equality:
$g( {f}^2 X, X) = -g({f} X, {f} X)\le 0$.
\end{proof}

\begin{Corollary}
 There are no nearly ${\cal C}$-manifolds with conditions \eqref{E-30b-xi} and \eqref{E-30-xi} which satisfy \eqref{2.3}.
\end{Corollary}

\begin{Theorem}\label{Th-4.1}
A weak nearly ${\cal C}$-manifold $(M^{2n+s}, {f},Q,{\xi_i},{\eta^i},g)$ satisfies
\begin{equation}\label{Eq-nabla-xi}
 \nabla\xi_i=0\quad(1\le i\le s)
\end{equation}
if and only if the manifold is locally isometric to the Riemannian product of a Euclidean $s$-space and a weak nearly K\"{a}hler manifold.
\end{Theorem}

\begin{proof}
For all vector fields $X, Y$ orthogonal to $\ker f$, we have
\begin{equation}\label{E-c-01b}
 2\,d\eta^i(X, Y) = g(\nabla_X\,\xi_i, Y) - g(\nabla_Y\,\xi_i, X).
\end{equation}
Thus, if the condition $\nabla\xi_i = 0$ holds, then the contact distribution $\cal D$ is integrable.
Moreover, any integral subma\-nifold of $\cal D$ is a totally geodesic submanifold.
Indeed, for $X,Y\perp\ker f$, we~have
 $g(\nabla_X\,Y, \xi_i)=-g(Y, \nabla_X\,\xi_i)=0$.
Since $\nabla_{\xi_i}\,\xi_j=0$, by de Rham Decomposition Theorem,
the mani\-fold is locally the Riemannian product $\bar M\times\mathbb{R}^s$.
The metric weak $f$-structure induces on $\bar M$ a weak almost-Hermitian structure, which, by these conditions, is weak nearly K\"{a}hler.

Conversely, if a weak nearly ${\cal C}$-manifold is locally the Riemannian pro\-duct $\bar M\times\mathbb{R}^s$, where $\bar M$ is a
weak nearly K\"{a}hler manifold and $\xi_i = (0, \partial_{y^i})$ (see also Example~\ref{Ex-02}), then $d\eta^j(X, Y)=0\ (X,Y\bot\ker f)$.
By~\eqref{E-c-01b} and $\nabla_{\xi_i}\,\xi_j=0$, we obtain~$\nabla\xi_i=0$.
\end{proof}

\begin{Corollary}
A nearly ${\cal C}$-manifold $(M^{2n+s}, {f},{\xi_i},{\eta^i},g)$ satisfies \eqref{Eq-nabla-xi}
if and only if the manifold is locally isometric to the Riemannian product of $\mathbb{R}^s$ and a nearly K\"{a}hler~manifold.
\end{Corollary}


\begin{Theorem}\label{Th-01}
Let a weak nearly ${\cal S}$-structure satisfy \eqref{E-30b-xi}, \eqref{E-30-xi}, and \eqref{E-nS-10};
then, the following is true:
\begin{enumerate}
\item[(i)]	
 The condition $\eta^j\circ {N}^{\,(1)}=0\ (1\le j\le s)$ yields $d\eta^j(X,Y)=\Phi(QX, Y)$ for all $j$.
\item[(ii)]	
 The condition \eqref{2.3} yields ${N}^{\,(1)}(X,Y) = 2\,\Phi(\widetilde QX, Y)\,\bar\xi$.
\end{enumerate}
\end{Theorem}

\begin{proof}
$(i)$ We calculate, using \eqref{4.NN}, \eqref{E-nS-01}, and $\eta^j\circ{f}=0$,
\begin{align*}
 \eta^j({N}^{\,(1)}(X,Y)) & - 2\,d\eta^j(X,Y) = \eta^j([{f},{f}](X,Y))
 \overset{\eqref{4.NN}}= \eta^j\big((\nabla_{{f} X}{f})Y -(\nabla_{{f}\,Y}{f})X \big)
  \\
 & \overset{\eqref{E-nS-01}}= \eta^j\big( (\nabla_{X}{f})\,{f}\,Y -(\nabla_{Y}{f})\,{f} X\big) + 4\,g(f^2 X, {f}\,Y) \\
 & = g\big( (\nabla_{X}{f}^2)\,Y -(\nabla_{Y}{f}^2)\,X, \,\xi_j\big) - 4\,g(Q X, {f}\,Y)\\
 & = (\nabla_X\,\eta^j)(Y) - (\nabla_Y\,\eta^j)(X)  - 4\,g(Q X, {f}\,Y)\\
 & = 2\,d\eta^j(X,Y) - 4\,g(Q X, {f}\,Y) .
\end{align*}
Here, we used the identity
 $2\,d\eta^j(X,Y) = (\nabla_X\,\eta^j)(Y) - (\nabla_Y\,\eta^j)(X)$.
\newline
Thus, if $\eta^i({N}^{\,(1)}(X,Y))=0$, then $d\,\eta^j(X,Y) = g(QX, fY)=\Phi(QX, Y)$ for all $j$.

\smallskip

$(ii)$ Using $d\,\Phi=0$, \eqref{2.1}, \eqref{E-3.3}, and \eqref{E-nS-01},
where $\bar\eta=\sum\nolimits_{\,i=1}^s\eta^i$ and $\bar\xi=\sum\nolimits_{\,i=1}^s\xi_i$,
we~get
\begin{align*}
 &\quad 3\,d\,\Phi(X,Y,Z) = - g((\nabla_X {f})Y, Z) +g((\nabla_Y {f})X, Z) -g((\nabla_Z {f})X, Y) \\
 & = -g((\nabla_X {f})Y, Z)
 +g\big(-(\nabla_X {f})Y +2\,g(fX, fY)\,\bar\xi + \bar\eta(X) f^2Y + \bar\eta(Y) f^2X,\ Z\big)\\
 &\quad +g\big( (\nabla_X {f})Z -2\,g(fX, fZ)\,\bar\xi - \bar\eta(X) f^2Z - \bar\eta(Z) f^2X,\ Y\big)\\
 &  = -3\,g((\nabla_X {f})Y, Z) -3\,g(f^2X, Y)\,\bar\eta(Z) +3\,g(f^2X, Z)\,\bar\eta(Y).
\end{align*}
Thus, \eqref{3.1A} holds.
Using \eqref{3.1A} in \eqref{4.NN} gives
\[
 [{{f}},{{f}}] = 2\,g(f^2X, fY)\,\bar\xi= -2\,g(QX, fY)\,\bar\xi=-2\,\Phi(QX, Y)\,\bar\xi,
\]
hence, ${N}^{\,(1)}(X,Y) = 2\,\Phi(\widetilde QX, Y)\,\bar\xi$.
\end{proof}

A consequence of Theorem~\ref{Th-01} is a rigidity result for ${\cal S}$-manifolds; see Theorem~1 of~\cite{AT-2017}.

\begin{Corollary}
A normal nearly ${\cal S}$-structure is an ${\cal S}$-structure.
\end{Corollary}


\section{Submanifolds of Weak Nearly K\"{a}hler Manifolds}
\label{sec:03-ns}

Here, we study weak nearly ${\cal S}$- and weak nearly ${\cal C}$- submanifolds in a weak nearly K\"{a}hler manifold.
The second fundamental form $h$ of a submanifold $M\subset (\bar M, \bar g)$
is related with $\overline\nabla$ (the Levi-Civita connection of $\bar g$ restricted to $M$)
and $\nabla$ (the Levi-Civita connection of metric $g$ induced on $M$ via the Gauss equation) by
\begin{equation}\label{Eq-NS-5}
 \overline\nabla_X Y = \nabla_X Y +h(X,Y)\quad (X,Y\in\mathfrak{X}_M).
\end{equation}
A submanifold is said to be {totally geodesic} if $h=0$.
The shape operator $A_N: X\mapsto -\overline\nabla_X N$ with respect to a unit normal $N$ is related with $h$ via the equalities
\begin{equation}\label{Eq-NS-5A}
 h_N(X,Y)=\bar g(h(X,Y),N)=g(A_N(X),Y)\quad (X,Y\in\mathfrak{X}_M).
\end{equation}

\begin{Lemma}\label{L-nS-01}
Let $(\bar M,\bar{f}, \bar g)$ be a weak Hermitian manifold
and $M^{2n+s}$ a submanifold of codimension $s$ equipped with
mutually orthogonal unit normals $N_i\ (i=1,\ldots,s)$ satisfying the condition
\begin{align}\label{E-NN}
 \bar g(\bar{f} N_i, N_j)=0\quad (1\le i,j\le s)
\end{align}
(trivial for $s=1$).
Then, $M$ inherits a metric weak $f$-structure $({f},Q,\xi_i,\eta^i,g)$ given by
\begin{align}\label{Eq-nS-01}
\notag
 & \xi_i= \bar{f}N_i,\quad \eta^i = \bar g(\bar{f}N_i,\,\cdot)\quad (i=1,\ldots,s),
 \quad g = \bar g|_M,\\
 & {f} = \bar{f} + \sum\nolimits_{\,i=1}^s \bar g(\bar{f}N_i, \,\cdot)\,N_i,\quad
 Q = -\bar{f}^{\,2} + \sum\nolimits_{\,i=1}^s \bar g(\bar{f}^{\,2} N_i,\,\cdot)\,N_i.
\end{align}
Moreover, \eqref{E-nS-10} holds on $M$ if $\bar{f}^{\,2} N_i\perp TM\ (1\le i\le s)$ and
\[
 ((\overline\nabla_X\,\bar{f}^{\,2})Y)^\top = 0\quad (X,Y\in\mathfrak{X}_M,\ Y\perp\ker f).
\]
\end{Lemma}

\begin{proof}
Using the skew-symmetry of $\bar{f}$ and \eqref{E-NN}, we verify \eqref{2.1}:
\begin{align*}
 &\quad {f}^{\,2} X = {f}(\bar{f} X - \sum\nolimits_{\,i=1}^s \bar g(\bar{f} X, N_i)\,N_i) \\
 & = \bar{f}\big(\bar{f} X - \sum\nolimits_{\,i=1}^s \bar g(\bar{f} X, N_i)\,N_i\big)
   -\bar g(\bar{f}(\bar{f} X - \sum\nolimits_{\,i,j=1}^s\bar g(\bar{f} X, N_i)\,N_i), N_j)\,N_j \\
 & = \bar{f}^{\,2} X -\sum\nolimits_{\,j}\bar g(\bar{f}^{\,2} N_j, X)\,N_j -\sum\nolimits_{\,i=1}^s \bar g(\bar{f} N_i, X)\,\bar{f} N_i
   + \sum\nolimits_{\,i,j=1}^s\bar g(\bar{f} X, N_i)\,\bar g(\bar{f} N_i, N_j)\,N_j \\
 & =-Q X +\sum\nolimits_{\,i=1}^s \eta^i(X)\,\xi_i\quad (X\in\mathfrak{X}_M).
\end{align*}
Since $\bar{f}^{\,2}$ is negative-definite, for nonzero $X\in\mathfrak{X}_M$ we obtain $\bar g(N_i, X)=0$ and
\[
 g(Q X, X) = \bar g(-\bar{f}^{\,2}X +\sum\nolimits_{\,i=1}^s \bar g(\bar{f}^{\,2}N_i, X)\,N_i, \,X)=-\bar g(\bar{f}^{\,2}X, X)>0,
\]
hence, the tensor $Q$ is positive-definite on $TM$.
Then, we calculate $(\nabla_X Q)Y$ for $X,Y\in\mathfrak{X}_M$ and $Y\perp\ker f$, using \eqref{Eq-NS-5} and \eqref{Eq-nS-01}
and the condition $\bar{f}^2 N_i\perp TM\ (1\le i\le s)$:
\begin{align*}
 &\quad (\nabla_X\,Q)Y = \nabla_X(Q Y) -Q (\nabla_X Y) \\
 & =\big\{\overline\nabla_X \big(-\bar{f}^2\,Y  +\sum\nolimits_{\,i=1}^s  g(\bar{f}^2 N_i, Y)N_i\big)
 - h(X,Q Y) +\bar{f}^2\big(\overline\nabla_X Y -h(X,Y)\big) \\
 & -\sum\nolimits_{\,i=1}^s  g\big(\bar{f}^2 N_i,\, \overline\nabla_X Y - h(X,Y)\big) N_i\big\}^\top \\
 & = (-(\overline\nabla_X (\bar{f}^2 Y)) +\bar{f}^2(\overline\nabla_X Y ))^\top -\sum\nolimits_{\,i=1}^s g\big(\bar{f}^2 N_i,\,Y)\,A_{N_i}X \\
 & = -((\overline\nabla_X\bar{f}^2) Y)^\top,
\end{align*}
where $^\top$ is the $TM$-component of a vector.
This completes the proof.
\end{proof}

The following theorem characterizes weak nearly ${\cal C}$- and weak nearly ${\cal S}$-submani\-folds
of a nearly K\"{a}hler manifold, using the
property of the second fundamental form.

\begin{Theorem}\label{Th-subm}
Let $(\bar M,\bar{f}, \bar g)$ be a weak nearly K\"{a}hler manifold and $M^{2n+s}$ a submanifold of codimension $s$ equipped with
mutually orthogonal unit normals $N_i\ (i=1,\ldots,s)$ satisfying \eqref{E-NN}.
If~the second fundamental form of $M$ and the induced metric weak $f$-structure $({f},Q,\xi_i,\eta^i,g)$ on $M$, given by \eqref{Eq-nS-01}, satisfy
\begin{align}
\label{Eq-NS-4}
\notag
 (i) \ h_{N_i}(X,Y) & = g(Q X,\, Y) + \sum\nolimits_{\,j,k=1}^s \big( h_{N_i}(\xi_j,\xi_k) - \delta_{j,k}\big)\,\eta^j(X)\,\eta^k(Y) ,\\
 (ii)\ h_{N_i}(X,Y) & = \sum\nolimits_{\,j,k=1}^s h_{N_i}(\xi_j,\xi_k)\,\eta^j(X)\,\eta^k(Y) ,
\end{align}
and
\begin{align}\label{E-AA-cond}
 h_{N_i}(\xi_j,\xi_k) & = h_{N_j}(\xi_i,\xi_k)\quad (1\le i,j\le s),
\end{align}
then $({f},Q,\xi_i,\eta^i,g)$ is
\begin{align}
(i)~\mbox{a weak nearly ${\cal S}$-structure};\quad
(ii)~\mbox{a weak nearly ${\cal C}$-structure}.
\end{align}
\end{Theorem}

\begin{proof}
Substituting
\[
 \bar{f}\,Y = {f}\,Y - \sum\nolimits_{\,i=1}^s\bar g(\bar{f}N_i, Y)\,N_i
 = {f}\,Y - \sum\nolimits_{\,i=1}^s \eta^i(Y)\,N_i
\]
in $(\overline\nabla_X\bar{f})Y$, where $X,Y\in\mathfrak{X}_M$, and using \eqref{Eq-NS-5} and Lemma~\ref{L-nS-01}, we~obtain
\begin{align*}
 (\overline\nabla_X\bar{f})Y = \overline\nabla_X(\bar{f}\,Y) - \bar{f}(\overline\nabla_X Y)
 & = (\nabla_X {f})Y +\sum\nolimits_{\,i=1}^s\big\{\eta^i(Y) A_{N_i}X - h_{N_i}(X,Y)\,\xi_i\big\} \\
 & +\sum\nolimits_{\,i=1}^s\big\{X(\eta^i(Y))-\eta^i(\nabla_X Y) +h_{N_i}(X, {f}Y)\big\} N_i .
\end{align*}
Thus, the $TM$-component of the weak nearly K\"{a}hler condition \eqref{Eq-NS-9}, using \eqref{Eq-NS-5} and \eqref{Eq-NS-5A}, takes the form
\begin{align}\label{Eq-NS-7}
\notag
 &\big((\overline\nabla_X\bar{f})Y + (\overline\nabla_Y\bar{f})X \big)^\top
 =(\nabla_X{f})Y +(\nabla_Y{f})X \\
 & +\sum\nolimits_{\,i=1}^s\big\{\eta^i(X) A_{N_i}Y +\eta^i(Y) A_{N_i}X -2\,h_{N_i}(X,Y)\,\xi_i\big\} =0.
\end{align}
Using \eqref{Eq-NS-5A}, one can show that \eqref{Eq-NS-4} is equivalent to the following:
\begin{align}\label{Eq-NS-4A}
\notag
 & (i) \ A_{N_i}X = -f^2 X + \sum\nolimits_{\,j,k=1}^s h_{N_i}(\xi_j,\xi_k)\,\eta^j(X)\,\xi_k ,\\
 & (ii)\ A_{N_i}X = \sum\nolimits_{\,j,k=1}^s h_{N_i}(\xi_j,\xi_k)\,\eta^j(X)\,\xi_k.
\end{align}

(i) If we have a weak nearly $\cal S$-structure, see \eqref{E-nS-01}, then from \eqref{Eq-NS-7} we~get
\begin{align}\label{Eq-NS-7A}
\notag
 & 2\,g(f X, f Y)\,\bar\xi +\bar\eta(Y)f^2 X +\bar\eta(X)f^2 Y \\
 & +\sum\nolimits_{\,i=1}^s\big\{\eta^i(X) A_{N_i}Y +\eta^i(Y) A_{N_i}X -2\,h_{N_i}(X,Y)\,\xi_i\big\} = 0,
\end{align}
Substituting the expressions of $h_{N_i}(X,Y)$ and $A_{N_i}$, see \eqref{Eq-NS-4}(i) and \eqref{Eq-NS-4A}(i), in \eqref{Eq-NS-7A}
and using \eqref{E-AA-cond} gives identity;
thus, we obtain a weak nearly ${\cal S}$-structure on $M$.

\smallskip

(ii) If we have a weak nearly $\cal C$-structure, see \eqref{E-nS-01b}, then from \eqref{Eq-NS-7} we~get
\begin{align}\label{Eq-NS-7B}
 \sum\nolimits_{\,i=1}^s\big\{\eta^i(X) A_{N_i}Y +\eta^i(Y) A_{N_i}X -2\,h_{N_i}(X,Y)\,\xi_i\big\} = 0.
\end{align}
Substituting the expressions of $h_{N_i}(X,Y)$ and $A_{N_i}$, see \eqref{Eq-NS-4}(ii) and \eqref{Eq-NS-4A}(ii), in \eqref{Eq-NS-7B}
and using \eqref{E-AA-cond} gives identity; thus, we obtain a weak nearly ${\cal C}$-structure on $M$.
\end{proof}

For $Q={\rm Id}$, the properties of \eqref{Eq-NS-4} lead us to the following.

\begin{Definition}\rm
A codimension $s$ submanifold $M^{2n+s}$ of a
Hermitian manifold $(\bar M,\bar{f}, \bar g)$,
equipped with mutually orthogonal unit normals $N_i\ (i=1,\ldots,s)$ satisfying
\begin{align}\label{Eq-NS-5B}
 h_{N_i}(X,Y) = a_i\,g(X, Y) + \sum\nolimits_{\,j,k=1}^s b_{i,j,k}\,\eta^j(X)\,\eta^k(Y) ,
\end{align}
where $a_i,b_{i,j,k}\in C^\infty(M)$ and $\eta^i\ (1\le i\le s)$ are linear independent one-forms on $M$,
will be called an $s$-\textit{quasi-umbilical} submanifold.
For $s=1$, condition \eqref{Eq-NS-5B} reads as follows, see~\cite{rov-122}:
\[
 h_N(X,Y) = a_1\,g(X,Y)+ b_1\,\eta(X)\,\eta(Y).
\]
The geometric meaning of \eqref{Eq-NS-5B} is that the restriction of $h_{N_i}$ on the distribution $\bigcap_{\,i=1}^s\ker\eta^i$ looks similar to
$h$ for totally umbilical submanifolds: $h=({\rm trace}_{g}\,h/\dim M)\,g$.
\end{Definition}

The following consequence of Theorem~\ref{Th-subm} extends the fact (see~Theorem 4.1 in~\cite{blair1976}) that a hypersurface of a nearly K\"{a}hler manifold is nearly Sasakian or nearly cosymplectic if and only if it is quasi-umbilical with respect to the almost contact form.

\begin{Corollary}
Let $(\bar M,\bar{f}, \bar g)$ be a nearly K\"{a}hler manifold and $M^{2n+s}$ a submanifold of codimension $s$ equipped with
mutually orthogonal unit normals $N_i\ (i=1,\ldots,s)$ satisfying \eqref{E-NN},
and $({f},\xi_i,\eta^i,g=\bar g|_M)$ the induced metric $f$-structure on $M$, given by
\begin{align*}
 \xi_i= \bar{f}N_i,\quad \eta^i = \bar g(\bar{f}N_i,\,\cdot)\quad (i=1,\ldots,s),
 \quad
 {f} = \bar{f} + \sum\nolimits_{\,j=1}^s\bar g(\bar{f}N_j, \,\cdot)\,N_j.
\end{align*}
If $M^{2n+s}$ is an $s$-\textit{quasi-umbilical} submanifold (with respect to the 1-forms $\eta^i$),
\begin{align*}
\notag
 (i) \ h_{N_i}(X,Y) & = g(X,\, Y) + \sum\nolimits_{\,j,k=1}^s \big( h_{N_i}(\xi_j,\xi_k) - \delta_{j,k}\big)\,\eta^j(X)\,\eta^k(Y) ,\\
 (ii)\ h_{N_i}(X,Y) & = \sum\nolimits_{\,j,k=1}^s h_{N_i}(\xi_j,\xi_k)\,\eta^j(X)\,\eta^k(Y) ,
\end{align*}
and \eqref{E-AA-cond} are true, then
$({f}, \xi_i, \eta^i, g)$
is
(i)~{a nearly ${\cal S}$-structure};\
(ii)~{a nearly ${\cal C}$-structure}.
\end{Corollary}



\section{Conclusions}

We have shown  that weak nearly $\cal S$- and weak nearly $\cal C$-structures are useful for studying
metric $f$-structures, e.g., totally geodesic foliations, Killing vector fields, and $s$-quasi-umbilical submanifolds.
Some classical results have been extended in this paper to weak nearly $\cal S$- and weak nearly $\cal C$-manifolds with additional conditions.
Based on the numerous applications of nearly K\"{a}hler, nearly Sasakian, and nearly cosymplectic structures,
we expect that weak nearly K\"{a}hler, $\cal S$- and $\cal C$-structures will be useful for
geometry and theoretical physics,
e.g., for NGT, the theory of $s$-cosymplectic structures and $s$-contact structures, multi-time Hamiltonian systems, and $s$-evolution~systems.

\end{document}